\documentclass[12pt,reqno]{amsart}
\usepackage{amssymb,amsthm}
\usepackage[utf8]{inputenc}
\usepackage[OT2,T1]{fontenc}

\advance\hoffset by -0.5truein
\DeclareSymbolFont{cyrletters}{OT2}{wncyr}{m}{n}
\DeclareMathSymbol{\Sha}{\mathalpha}{cyrletters}{"58}

\let\<\langle
\let\>\rangle

\def\Pois{\mathrm{Pois}}
\def\Lie{\mathrm{Lie}}
\def\Vecc{\mathrm{Vec}}

\def\GD {\mathop {\fam0 GD }\nolimits}
\def\SGD {\mathop {\fam0 SGD }\nolimits}

\def\wSGD{\mathrm{wSGD}}
\def\Com{\mathrm{Com}}
\def\Nov{\mathrm{Nov}}
\def\Der{\mathrm{Der}}
\def\wt{\mathop {\fam 0 wt}\nolimits}

\newtheorem{theorem}{Theorem}
\newtheorem{lemma}{Lemma}
\newtheorem{proposition}{Proposition}
\newtheorem{corollary}{Corollary}

\theoremstyle{definition}

\newtheorem{remark}{Remark}

\title[Gelfand--Dorfman algebras]{On the special identities of Gelfand--Dorfman algebras}
\author[P.S. Kolesnikov, B.K. Sartayev]{P.S. Kolesnikov$^{1)}$, 
B.K. Sartayev$^{1,2)}$}

\address{${}^{1)}$Sobolev Institute of Mathematics, \\
Akad. Koptyug prosp., 4\\
630090 Novosibirsk, Russia}

\address{${}^{2)}$Suleyman Demirel University, \\
Abylai khan street, 1/1 \\
040900 Kaskelen, Kazakhstan}

\thanks{This work was supported by the Program of fundamental scientific researches of the Siberian Branch of Russian Academy of Sciences, I.1.1, project 0314-2019-0001. 
The second author was supported by grant AP08052405 of Ministry of Education and Science of Republic of Kazakhstan}

\begin{document}

\begin{abstract}
A Gelfand--Dorfman algebra (GD-algebra) is said to be special if it can be embedded into a differential Poisson algebra.
In this paper, we prove that the class of all special GD-algebras is closed with respect to homomorphisms and thus forms a variety. 
We describe a technique for finding potentially all special identities of GD-algebras and derive two known special identities of degree 4 in this way.
By computing the Gr\"obner basis for the corresponding shuffle operad, we show that these two identities imply all special ones up to degree~5.
\end{abstract}

\maketitle

\section*{Introduction}

A linear space $V$ with two bilinear operations $ \circ $ and $[\cdot,\cdot]$ is called 
a {\em Gelfand--Dorfman algebra} (or simply $\GD$-algebra) if $(V,\circ)$ is a 
Novikov algebra, $(V,[\cdot,\cdot])$ is a Lie algebra, and the following 
additional identity holds:
\begin{gather}\label{eq:GD1}
b\circ [a,c]=[a,b\circ c]-[c,b\circ a]+[b,a]\circ c-[b,c]\circ a.
\end{gather}
Recall that the variety of Novikov algebras is defined by the following identities:
\begin{gather}
(a\circ b)\circ c-a\circ (b\circ c)=(b\circ a)\circ c-b\circ (a\circ c), \label{eq:LSymm} \\
(a\circ b)\circ c=(a\circ c)\circ b. \label{eq:RComm}
\end{gather}

The axioms above emerged in the paper \cite{GD83} as a tool for constructing Hamiltonian operators 
in formal calculus of variations. Later it was shown \cite{Xu2000} that $\GD$-algebras are in one-to-one
correspondence with quadratic Lie conformal algebras playing an important role in the theory
of vertex operators. 
The class of $\GD$-algebras is governed by a binary quadratic operad (\cite{GinzburgKapranov}, 
see also \cite{bremner-dotsenko}
for the definition) denoted $\GD $. As it was shown in \cite{KSO2019}, 
the Koszul dual operad $\GD^!$ corresponds to the class 
of differential Novikov--Poisson algebras introduced in \cite{BCZ2018-b}. The latter algebras 
play an important role in the combinatorics of derived operations on non-associative 
algebras \cite{KSO2019}. 

In the present paper, we study special $\GD$-algebras, i.e., those embeddable 
into Poisson algebras with a derivation. We prove in section~\ref{sec:Special}
that the class of special $\GD$-algebras 
is closed under homomorphic images and thus forms a variety. 
Non-special GD-algebras exist: the examples were found implicitly in \cite{KSO2019} 
and explicitly  in \cite{KolPan2019}. Note that all these examples 
are of dimension~three. We apply the Gr\"obner--Shirshov bases 
technique for Poisson algebras to prove that all 2-dimensional GD-algebras are special.

In section \ref{identities}, we give a technical method to find all special identities of $\GD$-algebras and explicitly find all special identities of GD-algebras up to degree $5$. 
In section \ref{programm}, we prove that all special identities up to degree $5$ are consequences of two independent special identities of degree $4$
found in \cite{KSO2019}. For that purpose, we convert the symmetric operad GD with special identities of degree $4$ into a shuffle operad \cite{bremner-dotsenko}
and compute the first five components of its Gr\"obner basis \cite{DotKhor2010} by means of the  computer algebra package \cite{DotsHij}.
In the same way, we calculate the dimensions of GD operad up to degree 6.

\section{Special GD-algerbas}\label{sec:Special}

Suppose $(P,\cdot,\{\cdot,\cdot\})$ is a Poisson algebra with a derivation~$d$.
Hereinafter, we will denote $d(x)$ by $x'$ for simplicity.
Hence, $(P,\cdot )$ is an associative and commutative algebra, $(P,\{\cdot,\cdot\})$
is a Lie algebra and the {\em Leibniz identity} holds:
\[
 \{x,yz\} = y\{x,z\}+z\{x,y\},\quad x,y,z\in P.
\]
The linear map $d:P\to P$ acts as a derivation relative to both operations:
\[
 (xy)'=xy'+x'y,\quad \{x,y\}' = \{x',y\}+\{x,y'\}, \quad x,y\in P.
\]
Then the same space $P$ equipped with operations
\[
x\circ y=x y', \quad [x,y]=\{x,y\},
\] 
is a GD-algebra denoted $P^{(d)}$ \cite{Xu2002}.
A GD-algebra $V$ is said to be {\em special} if it can be embedded into a 
GD-algebra $P^{(d)}$ for an appropriate differential Poisson algebra~$P$.

It is clear that the class of special GD-algebras is closed with respect to 
subalgebras and Cartesian products, so the class of all homomorphic images 
of special GD-algebras is a variety denoted $\SGD$.

The relation between differential Poisson algebras and GD-algebras described above 
is similar to the well-known relation between 
associative and Jordan algebras \cite{Cohn} (see also \cite[Chapter~3]{ZSSS}). 
However, in contrast to the Jordan algebra case, the class of special 
GD-algebras turns to be closed with respect to homomorphisms.

\begin{theorem}\label{thm:SpecialImage}
Let $V$ be a special GD-algebra. 
Then every homomorphic image of $V$ is special.
\end{theorem}

\begin{proof}
As $V$ is a special GD-algebra, it is a homomorphic image of 
the free special GD-algebra $\SGD\<X \>$ generated by an appropriate set~$X$. 
Therefore, it is enough to show that for every set $X$ all homomorphic images of 
$\SGD\<X\>$ are special.

Let us recall the structure of $\SGD\<X\>$ \cite{KSO2019}. Consider 
the free differential Poisson algebra 
$\Pois\Der\<X,d\>$ generated by the set $X$. 
The operation set of this system 
consists of a commutative multiplication, Poisson bracket 
$\{\cdot,\cdot \}$, and a unary operation $d$ which acts as a derivation with 
respect to both binary operations. 
As above, we will write $u'$ for $d(u)$.
Define the {\em weight} of a monomial from $\Pois\Der\<X,d\>$ by induction as 
follows:
\[
\begin{gathered}
\wt(x)=-1,\quad x\in X; \quad \wt(u') = \wt(u)+1; \\
\wt(\{u,v\})=\wt(u)+\wt(v)+1; \quad \wt(uv)=\wt(u)+\wt(v).
\end{gathered}
\]
For a weight-homogeneous polynomial  $f\in \Pois\Der\<X,d\>$ 
we denote by $\wt(f)$ the weight of its monomials.

As shown in \cite{KSO2019},
$\SGD\<X\>$ is isomorphic to the subspace of 
$\Pois\Der\<X,d\>^{(d)}$ span\-ned by all monomials of 
weight~$-1$.

Suppose $I$ is an ideal of the GD-algebra $\SGD\<X\>$, the latter is immersed into $ \Pois\Der\<X,d\>$.
Consider the ideal $\hat I$ of 
$\Pois\Der\<X,d\>$ generated by~$I$. 
It is enough to show 
$\hat I \cap \SGD\<X\> = I$. If the latter holds, then
$\Pois\Der\<X,d\>/\hat I$ is a differential Poisson envelope 
of the GD-algebra $V =\SGD\<X\>/I$. 

An arbitrary element in $f\in \hat I$ may be obtained as
\[
f= \sum_i F_i(X,t)|_{t=u_i}, 
\]
where $F_i(X,t)\in \Pois\Der\<X\cup\{t\}, d\>$, 
$u_i\in I\subset \SGD\<X\>\subset \Pois\Der\<X,d\>$.
Note that $\wt(f)=\wt(u_i)=-1$,
so we may assume 
$\wt (F_i) = -1$ for all~$i$ (otherwise, it is enough to choose 
the homogeneous components of weight $-1$ for each $F_i$).
Therefore, 
\[
F_i(X,t) \in \SGD \<X\cup \{t\} \>, 
\]
and 
$F_i(X,u_i) \in I$. Hence, $f\in I$.
\end{proof}

\begin{corollary}\label{cor:SGD-Var}
 The class of special Gelfand--Dorfman algebras forms a variety.
\end{corollary}

\begin{remark}
Novikov algebras are particular cases of GD-algebras (with $[x,y]=0$). 
It was shown in \cite{dzhuma} that the free Novikov algebra embeds into 
the free differential commutative algebra. Hence, restricting the proof 
of Theorem \ref{thm:SpecialImage} to Novikov algebras we obtain another 
proof of speciality of all Novikov algebras \cite{BCZ2017}.
\end{remark}

The examples of non-special GD-algebras  were constructed in \cite{KSO2019}, \cite{KolPan2019}. 
It is worth mentioning that all these examples are of 
dimension~3.
In the next statement we show that dimension 3 is a minimal one: all 2-dimensional 
GD-algebras are special. 

\begin{theorem}\label{thm:2dimSpecial}
Let $V$ be a 2-dimensional GD-algebra. Then $V$ is special.
\end{theorem}

\begin{proof}
Let $X=\{u,v\}$ be a basis of a 2-dimensional GD-algebra $V$. 
If $[u,v]=0$ then $V$ is a pure Novikov algebra and by \cite{BCZ2017} $V$ embeds 
into the 
commutative differential algebra $\Com\Der\<X\>/(xy'-x\circ y)$
which may be considered as a Poisson one with respect to the
trivial bracket.

If $V$ is not abelian as a Lie algebra then we may assume $[u,v]=v$.
It is straightforward do deduce from \eqref{eq:GD1}, \eqref{eq:LSymm},
and \eqref{eq:RComm} that the multiplication table for $\circ $ on $V$ has the 
following form:
\[
\begin{gathered}
u\circ u=\alpha u+\delta v, \\
v\circ v=0, \\
u\circ v=\gamma v, \\
v\circ u=\alpha v,
\end{gathered}
\]
for $\alpha,\gamma,\delta \in \Bbbk $.
It is enough to show that $V$ can be embedded into a differential Poisson 
algebra. We consider three cases:
(1) $\alpha\neq\gamma$, 
(2) $\alpha=\gamma \neq 0$,  
(3) $\alpha=\gamma=0$.

Case 1: $\alpha\neq\gamma$.
Note that 
$u\circ v-v\circ u=(\gamma-\alpha)v=(\gamma-\alpha)[u,v]$,
i.e., 
\[
[u,v]={\frac{1}{\gamma-\alpha}}(u\circ v-v\circ u).
\]
Therefore, the structure of a GD-algebra on $V$ is almost completely similar to
the commutator GD-algebra on a Novikov algebra considered in \cite{KolPan2019}. 
The only difference is the multiplicative constant $\frac{1}{\gamma-\alpha}$ 
which does not affect the construction of a differential Poisson envelope.  
Namely, consider the free differential commutative algebra 
$F=\Com\Der\<\{u,v\},d\>=\Com\<u,v,u',v',u'',v'',\dots  \>$, 
and define a bracket $\{\cdot,\cdot\}$ on $F$ as follows.
On the generators, let
\begin{equation}\label{eq:Bracket1}
\{u^{(m)},v^{(n)}\}={\frac{1}{\gamma-\alpha}}
 ((n-1)u^{(m+1)}v^{(n)}-(m-1)u^{(m)}v^{(n+1)}), \quad 
 n,m\ge 0.
\end{equation}
Then apply the Leibniz rule to expand the bracket on the entire~$F$.
The bracket obtained is proportional to the one considered in 
\cite{KolPan2019}. Therefore, the operation
\eqref{eq:Bracket1} 
satisfies the Jacobi identity, and
the operation $d$ is a derivation with respect to this bracket. 
Consider the ideal $I_V$ of the free differential commutative algebra $F$ 
generated by $xy'-x\circ y$, $x,y\in \{u,v\}$.
The ideal $I_V$
is invariant under all operations $\{f,\cdot\}$, $f\in F$.
Hence, the universal differential commutative envelope 
$F/I_V$ of the Novikov algebra $V$ is also a differential 
Poisson envelope of $V$ as of a GD-algebra. 

Case 2: $\alpha=\gamma\neq 0$. An obvious change of variables 
$\frac{1}{\alpha}u -\frac{\delta}{\alpha^2}v \to u$ ($v\to v$) leads to the 
following multiplication table in $V$:
\begin{equation}\label{eq:Mult_Table-2}
{}
[u,v]=\frac{1}{\alpha} v, \quad
u\circ u=u, \quad u\circ v=v, \quad v\circ u=v, \quad v\circ v=0.
\end{equation}

Consider the polynomial algebra $\Bbbk [x,e]$ equipped with two commuting 
derivations 
\[
d_1 = \dfrac{\partial}{\partial x},\quad 
d_2 = \dfrac{e}{\alpha }\dfrac{\partial }{\partial e}.
\]
Then $\{f,g\} = d_1(f)d_2(g)-d_2(f)d_1(g)$, 
$f,g\in \Bbbk [x,e]$,
is a Poisson bracket on $\Bbbk[x,e]$, and $d_1$ is a derivation on the Poisson 
algebra obtained.
In particular,
\begin{equation}\label{eq:Bracket2}
\{x,e\} = \frac{1}{\alpha } e.
\end{equation}
Denote by $A$ the quotient of $\Bbbk [x,e]$ modulo the ideal $I$ generated by 
$e^2$. As $d_1(I), d_2(I)\subseteq I$, the commutative algebra $A$
is a differential Poisson algebra generated by $x$ and $e$ with a bracket 
defined via \eqref{eq:Bracket2} and with a derivation 
\[
d(x)=1,\quad d(e)=0.
\]
Now, the initial GD-algebra $V$ embeds into $A^{(d)}$ by
\[
u \mapsto x,\quad v\mapsto ex.
\]
Indeed, it is straightforward to check that this 
map preserves the multiplication table \eqref{eq:Mult_Table-2}.

Case 3: $\alpha=\gamma=0$.
If $\delta=0$ then $V$ is a pure Lie algebra and thus embeds 
into the associated graded Poisson algebra $P(V)=\mathrm{gr}\,U(V)$ of the 
universal associative envelope of~$V$ with zero derivation. 

If $\delta\neq 0$ then, up to a change of variables, we may assume
\[
[u,v]=v,\quad
u\circ u=v,\quad v\circ v=u\circ v=v\circ u=0.
\]

Consider the polynomial algebra $F=\Bbbk [u,v,u',v']$
in four formal variables. Define a skew-symmetric 
bracket $\{\cdot ,\cdot \}$ on $F$ by the Leibniz rule
starting from
\begin{equation}\label{eq:GSB-2}
\begin{gathered}
\{u,v\}=v, \quad \{u,u'\}=u',\quad \{u,v'\}=2v', \\
\{v,u'\}=v',\quad \{v,v'\}=\{u',v'\}=0.
\end{gathered}
\end{equation}
This is a Poisson bracket since it is straightforward 
to check the Jacobi identity on the generators.

Let $I$ be the ideal of $F$ generated by 
\[
uu'-v,\ uv',\ vu',\ vv',\ vv,\ u'u'-v',\ u'v',\ v'v'.
\]
These polynomials form a Gr\"obner basis in $F=\Bbbk [u,v,u',v']$, 
so the quotient $A=F/I$ contains $V$ as a subspace. 
The ideal $I$ is closed under applying the bracket $\{x,\cdot \}$ for 
$x=u,v,u',v'$. Hence, $A=F/I$ is a Poisson algebra.

The ideal $I$ is closed under the derivation $d$ of $F$ (as of commutative algebra) defined by
\[
d(u)=u',\quad d(v)=v',\quad d(u')=0,\quad d(v')=0.
\]
Therefore, $d$ induces a derivation on $A$.

It remains to check that $d$ is a derivation 
relative to the Poisson bracket on $A$,
i.e., $d\{f,g\} =  \{d(f),g\} + \{f,d(g)\}$ for  $f,g\in A$.
 Note that the linear basis of $A$ consists of reduced monomials relative 
 to the Gr\"obner basis \eqref{eq:GSB-2}: $1,u',v',u^n, u^mv$, for $n,m\ge 0$.
By definition, 
\[
 d(u^n) = \begin{cases}
           0,& n=0,\\
           u',& n=1, \\
           nu^{n-2}v, & n\ge 2,
          \end{cases}
\quad 
 d(u^mv) = \begin{cases}
            v', & m=0, \\
            0, & m\ge 1.
           \end{cases}
\]
If $f=u$, $g=v$ then 
\[
 \{d(f),g\}+\{f,d(g)\} = \{u',v\}+\{u,v'\}=-v'+2v'=v' = d\{f,g\}.
\]
If $f=u^n$, $n>1$, $g=v$ then
\begin{multline}\nonumber
 \{d(f),g\}+\{f,d(g)\} = \{nu^{n-2}v,v\}+\{u^nv',v'\}\\
  =n(n-2)u^{n-3} vv + 2nu^{n-1}v'v'= 0= d\{f,g\}.
\end{multline}
For the remaining pairs of basic elements $f$, $g$, the derivation property 
for $d$ may be checked in a similar way.

Hence, $A$ is a differential Poisson algebra. 
It follows from the definition of $d$ and $A$ that 
$V$ as a GD-algebra embeds into $A^{(d)}$.
\end{proof}

\section{Special identities of GD-algebras}\label{identities}

Since both classes $\GD $ and $\SGD$ are varieties and thus defined 
by some sets of identities, the interesting problem is 
to determine a list of those identities that hold in 
$\SGD$ but do not hold in $\GD $. As in the case of Jordan algebras, let us call such identities {\em special}.

In \cite{KSO2019}, the following description 
of $\SGD $ was proposed: the operad of special 
$\GD$-algebras is a sub-operad in the Manin white product $\GD^!\circ \Pois$. 
Here $\Pois $ is the operad of Poisson algebras 
and $\GD^! $ is Koszul dual to $\GD $.
However, finding special identities 
in this way is technically hard since the 
entire operad $\GD^!\circ \Pois$ is pretty large.

In this section, we develop another approach which 
allows us to show that the two identities found in 
\cite{KSO2019} exhaust all special identities of degree~4
and find the complete list of special identities of degree~5.
In the next section, we apply Gr\"obner basis technique 
for operads to show that the identities of degree 5 found in this section follow from the identities of degree~4.

Let $X=\{x_1,x_2,\dots \}$ be a countable set,
and let
$\GD\<X\>$ be the free GD-algebra generated by~$X$.
Suppose $B$ is a linear basis of $\GD\<X\>$, $X\subset B$, equipped with a linear order $\le $.
Note that finding the explicit form 
of $B$, or at least of the multilinear 
part of $B$, is a separate interesting problem.
Let us construct $U=\Pois\Der\<B,d\>/I$, where 
$I$ is the differential ideal generated by
\[
ad(b)-a\circ b,\quad a,b\in B,
\]
\[
\{ a,b\} - [a,b],\quad a,b\in B, \ a>b.
\]
The algebra $U$ constructed is the universal Poisson differential envelope of $\GD\<X\>$.
The kernel of the natural homomorphism
$\tau: \GD\<X\> \to U$, $\tau(b)= b+I$, 
is exactly the set of special identities of $\GD$-algebras.

Denote $B^{(\omega )} = \{b^{(n)} \mid b\in B, n\ge 0 \}$
and expand the order $\le $ to $B^{(\omega )}$ by the following rule:
\begin{equation}\label{eq:Order_diff}
b^{(n)}<a^{(m)} \iff (b,-n)<(a,-m)\ \text{lexicographically}.
\end{equation}
The ordering is motivated by the Shirshov's argument \cite{Shir62}
concerning the standard bracketing on Lyndon--Shirshov words.

Let us identify $\Pois\Der\<B,d\>$ with $\Pois\<B^{(\omega )}\>$ assuming $d^n(b)=b^{(n)}$, 
then $I$ coincides with the ideal in 
$\Pois\<B^{(\omega )}\>$
generated by
\begin{gather}
ab^{(n)}-(a\circ b)^{(n-1)}+\sum_{i=1}^{n-1} \binom{n-1}{i} a^{(i)} b^{(n-i)},  \ n\ge1,
                  \label{eq:IdealGenerators1}
\\
\{a,b^{(n)}\} -[a,b]^{(n)} + \sum\limits_{i=1}^{n} \binom{n}{i} \{a^{(i)},b^{(n-i)} \},\ a>b,\ n\ge 0.
                        \label{eq:IdealGenerators2}
\end{gather}
In order to find the kernel of $\tau $ 
it is enough to calculate the intersection 
of $I$ with $\Bbbk B^{(0)}$. 
The latter can be done if we know 
the Gr\"obner--Shirshov basis (GSB) of $I$ in the free 
Poisson algebra $\Pois\<B^{(\omega )}\>$ \cite{BCZ2019}.

Let us proceed as follows. Consider the Lie algebra 
$L= \Lie \<B^{(\omega )} \>/J$, where $J$ is the ideal 
generated by \eqref{eq:IdealGenerators2}. 
Then $\Pois\<B^{(\omega )}\>/I$
is  the same as 
the quotient of the symmetric algebra 
$S(L)$ modulo the ideal 
generated by all elements of the form
\begin{equation}\label{eq:IdealGen-Pois}
\{a_1^{(n_1)},\{a_2^{(n_2)}, \dots , \{a_k^{(n_k)}, f\}\dots \}\}, 
\end{equation}
for all  $f$ in \eqref{eq:IdealGenerators1},
$a_i\in B$, $n_i\ge 0$.
Thus it is enough to calculate the intersection 
of $L$ with the Gr\"obner basis of an ideal 
in $S(L)$ generated by relations \eqref{eq:IdealGen-Pois}.
Let us first find the Gr\"obner--Shirshov basis 
of the Lie algebra~$L$ in a slightly more general 
context.

\begin{lemma}\label{lem:DiffGSB}
Let $\mathfrak g$ be a Lie algebra with a linearly ordered
basis $B$. Then \eqref{eq:IdealGenerators2} is 
a Gr\"obner--Shirshov basis in 
$\Lie \langle B^{(\omega )}\rangle $ relative to the deg-lex ordering based on \eqref{eq:Order_diff}.
\end{lemma}

\begin{remark}
Given a Lie algebra $\mathfrak g$ with a linear 
basis $B$ as above, 
the Lie algebra generated by $B^{(\omega )}$ with 
defining relations \eqref{eq:IdealGenerators2} is 
the universal differential envelope of~$\mathfrak g$.
\end{remark}

\begin{proof}
One may simply check that all compositions of intersection 
of the relations \eqref{eq:IdealGenerators2} are trivial 
in the sense of \cite{Shir62} (see also \cite{BokChen14}).
Note that a specific technique for calculating 
Gr\"obner--Shirshov bases in differential Lie algebras 
was proposed in \cite{CCLiu2009}, but we have a different ordering.

A more conceptual way is based on the following observation. 
Denote by $L$ the quotient of 
$\Lie \langle B^{(\omega )}\rangle $ 
by the Lie ideal generated by \eqref{eq:IdealGenerators2}.
The multiplication table on $\mathfrak g$ corresponds to $n=0$.
Add a new letter $t$ to the alphabet $B$ (assuming $t<B$) and construct 
\[
U = \Lie \<B,t \mid \{a,b\} - [a,b], \ a>b \>.
\]
Since the multiplication table is always a Gr\"obner--Shirshov basis, the linear basis of $U$ consists of all Lyndon--Shirshov 
words $[u]$ in $B\cup \{t\}$ such that their associative images 
$u$ do not contain subwords $ab$ for $a>b$.
Such a word $u$ is either equal to $t$ or may be written as 
\[
u = u_1a_1t^{k_1}\dots u_2a_2t^{k_2} \dots u_ma_mt^{k_m},
\]
for $u_i\in B^*$, $a_i\in B$, $k_i\ge 1$.
As shown in \cite{Shir62}, the standard bracketing $[u]$ on $u$ 
is constructed in the same way as on 
\[
u_1a_1^{(k_1)}\dots u_2a_2^{(k_2)}\dots u_ma_m^{(k_m)}, 
\]
where $a^{(k)} = \{\{\dots \{a,t\},\dots , t\}, t \}$
and the order is defined by \eqref{eq:Order_diff}.
Hence, the linear basis of $U$ consists of $t$ and of all those 
Lyndon--Shirshov words that are reduced modulo \eqref{eq:IdealGenerators2}.
Therefore, the latter relations form a Gr\"obner--Shirshov basis 
in $\Lie \<B^{(\omega )} \>$, and, in particular, 
$U \simeq  \Bbbk t \ltimes L$,
where $\{f,t\} = f'$ for $f\in L$.
\end{proof}

\begin{corollary}\label{cor:LBasis}
The linear basis of $L$ consists of all nonassociative Lyndon--Shirshov words of the form 
\[
[x_{11}\dots x_{1l_1}a_1^{(k_1)} \dots x_{m1}\dots x_{ml_m} a_m^{(k_m)}],
\quad k_i\ge 1,\ x_{i1}\le \dots \le x_{il_i}\le a_{i}.
\]
\end{corollary}

In order to find the intersection of $L$ with the ideal in $S(L)$
generated by \eqref{eq:IdealGen-Pois} let us define a rewriting system in $S(\Lie \<B^{(\omega )}\>)$ based on the relations \eqref{eq:IdealGenerators2}
with principal parts $\{a,b^{(n)}\}$, $a>b$, 
and on the relations
\eqref{eq:IdealGen-Pois}
by choosing the principal parts 
as 
\[
a LS(\{a_1^{(n_1)},\{a_2^{(n_2)}, \dots , \{a_k^{(n_k)}, b^{(n)}\}\dots \}\}), \quad n_i\ge 0, \ n\ge 1,
\]
where $LS(u)$ stands for the principal Lyndon--Shirshov word 
in the expansion of $u\in \Lie \<B^{(\omega )}\>$.

For example, if 
$f = ac'-a\circ c$, 
$b> c$,
then the relation
$\{b,f\} = \{b,ac'\}-\{b,a\circ c \}$ from 
\eqref{eq:IdealGen-Pois}
gives rise to the following rewriting rule:
\[
a\{b,c'\} \to [a,b]\circ c + [b,a\circ c].
\]
Similarly, $\{c,ab'-a\circ b\}$ gives us 
\[
a\{c,b'\} \to [a,c]\circ b + [c,a\circ b].
\]
On the other hand, we have a rule 
\[
\{b,c'\} \to \{c,b'\} + [b,c]'.
\]
from \eqref{eq:IdealGenerators2}.

Here we have an ambiguity in $a\{b,c'\}$: there are
two ways how to rewrite it.
This particular critical pair 
is convergent by \eqref{eq:GD1}:
\[
a\{b,c'\} \to a\{c,b'\} + a[b,c]'\to 
[a,c]\circ b + [c,a\circ b] + a\circ [b,c] = 
[a,b]\circ c + [b,a\circ c].
\]

Denote by $\mathcal G(B)$ the oriented graph 
with vertices $S(\Lie \<B^{(\omega )}\>)$ and 
edges defined by the rewriting rules based 
on \eqref{eq:IdealGenerators2} and \eqref{eq:IdealGen-Pois}.
These rules preserve the weight ($\wt$) of differential
polynomials in $B^{(\omega )}$ as well as the degree in~$X$.
The set $\Bbbk B$ is homogeneous of weight $-1$. 
Hence, $\mathcal G(B)$ splits into connected components $\mathcal G_{n,w}$ that contain vertices of degree $n$ in $X$ of weight~$w$.
Denote $\mathcal G_{n,-1}$ by $\mathcal G_n(B)$.
This graph has no infinite chains (i.e., it is a rewriting system):
as shown in \cite{KSO2019}, 
these rules applied to a differential 
monomial of weight $-1$ rewrite it to an element of 
$\Bbbk B$ in a finite number of steps.
This observation allows us
not to define any ordering on commutative monomials in 
$S(\Lie \<B^{(\omega )}\>)$.

In order to find the special identities 
of degree $n$ (with respect to $X$) it is enough to 
find the relations in $\Bbbk B$ that make the rewriting 
system $\mathcal G_{n}(B)$ confluent.

Given a fixed integer $n\ge 3$, 
it is enough to consider all multilinear differential
 monomials of weight $-1$ and of degree $n$
in $X$, and expand all critical pairs 
in the rewriting system $\mathcal G_n(B)$.
Note that all such pairs without brackets (i.e., those from $\Com \<B^{(\omega )}\rangle $) are confluent
by \cite{BCZ2017} and we do not need to consider them. Similarly, 
Lemma~\ref{lem:DiffGSB} shows that all ``pure Lie'' critical pairs
$f \leftarrow u\to g$, $u\in \Lie \<B^{(\omega )} \>$,
are trivial.

For $n=3$, the remaining critical pairs 
correspond to the ambiguities of the form 
$a\{b,c'\}$, $b>c$. 
As shown above, these critical pairs are convergent.

For $n=4$, there are five potential ambiguities:
\begin{itemize}
    \item[(A1)] $a\{b,\{c,d'\}\}$, $c>d$;
    \item[(A2)] $a\{\{b,c'\},d\}$, $c\le d$, $b>c$; 
    \item[(A3)] $ab'\{c,d'\}$, $c\ge d$;
    \item[(A4)] $ab\{c',d'\}$, $c>d$;
    \item[(A5)] $ab \{c,d''\}$, $c\ge d$.
\end{itemize}

In the case (A1), one may apply either
$\{c,d'\} \to \{c',d\} + [c,d]'$
or the rule coming from 
$\{b,\{c, ad'-a\circ d \}\}$ in \eqref{eq:IdealGen-Pois}:
\[
a\{b,\{c,d'\}\} \to f_1= a\{b,\{c',d\}\} + a\{b,[c,d]'\}
\]
or
\[
a\{b,\{c,d'\}\} \to f_2 = \{a,b\}\{c,d'\} + \{a,c\}\{b,d'\} - [b,[c,a]]d' +\{b,\{c,a\circ d\}\}.
\]
Then apply the rewriting rules 
coming from \eqref{eq:IdealGen-Pois} to get
\[
f_2\to [c, [a,b]\circ d] - [c,[a,b]]\circ d
 + [b,[a,c]\circ d] +[b,[c,a\circ d]]
\]
and, similarly,
\[
f_1\to -[d, [a,b]\circ c] + [d,[a,b]]\circ c
 - [b,[a,d]\circ c] - [b,[d,a\circ c]] +[b, a\circ [c,d]] 
 - [b,a]\circ [c,d]
\]
One may see that 
$f_1=f_2$ due to \eqref{eq:GD1}, namely, 
$f_2-f_1$ is zero modulo the Gel'fand--Dorfman relations 
at $(a,c,d)$ and $(c,[a,b], d)$.

The ambiguity (A2) also gives rise to a 
convergent critical pair modulo
\eqref{eq:GD1}, as one may
check in a similar way.

In the case (A3), we have three possible rewriting rules:
\[
\begin{aligned}
\{c,d'\} &{}\to \{c',d\} + [c,d]',\quad \text{if } c>d, \\
a\{c,d'\} &{}\to [c, a\circ d] + [a,c]\circ d, \\
ab' &{} \to a\circ b.
\end{aligned}
\]
The first one, combined with either of other rules, leads 
to a convergent critical pair due to \eqref{eq:GD1}.
Consider the pair coming from the last two rules:
\[
\begin{gathered}
ab'\{c,d'\} \to [c,a\circ d]\circ b + ([a,c]\circ d)\circ b, \\
ab'\{c,d'\} \to (a\circ b)\{c,d'\} \to [c,(a\circ b)\circ d] - [c,a\circ b]\circ d.
\end{gathered}
\]
The relation obtained
\begin{equation}\label{eq:Spec1}
[c,a\circ d]\circ b + ([a,c]\circ d)\circ b = 
 [c,(a\circ b)\circ d] - [c,a\circ b]\circ d
\end{equation}
is one of the special identities found in \cite{KSO2019}.

For (A4), it is enough to consider the following rules:
\[
\begin{aligned}
a\{c',d'\} &{} \to \{a,c'\}d' + \{c',a\circ d\}, \\
b\{c',d'\} &{}\to \{b,c'\}d' + \{c',b\circ d\}, \\
b\{c',d'\} &{}\to -\{b,d'\}c' - \{d',b\circ c\}.
\end{aligned}
\]
For example, the second rule allows us to rewrite 
$ab\{c',d'\}$ as 
\[
\{b, (a\circ d)\circ c \} - \{b, a\circ d\}\circ c + \{a\circ c, b\circ d\} -\{a,b\circ d\}\circ c.
\]
Now use \eqref{eq:Spec1} to replace the first term with 
$[b, (a\circ c)]\circ d - ([b,a]\circ c)\circ d + [b, a\circ d]\circ c$. As a result, we obtain 
\begin{equation}\label{eq:A4case}
ab\{c',d'\} \to 
[b,a\circ c]\circ d - ([b,a]\circ c)\circ d + [a\circ c, b\circ d]
- [a,b\circ d]\circ c
\end{equation}

The convergence of the corresponding three critical pairs starting 
at $ab\{c',d'\}$ is equivalent to the conditions that 
the right-hand side of \eqref{eq:A4case}
is symmetric with respect to the permutation $(a,b)$ and skew-symmetric by $(c,d)$.

Either of these two conditions (due to \eqref{eq:RComm} and anti-commutativity of $[\cdot,\cdot ]$) 
gives rise to the following relation:
\begin{multline}\label{eq:Spec2}
2([a,b]\circ c)\circ d 
=
[b\circ c, a\circ d] - [a\circ c, b\circ d] \\
+ ([a,b\circ c] -[b,a\circ c])\circ d  
+ ([a,b\circ d] - [b,a\circ d])\circ c.
\end{multline}
This is another special identity found in \cite{KSO2019}.

In the case (A5), we have two types of critical pairs:
the first one coming from the rules based on 
$\{c,bd''-(b\circ d)' + b'd'\}$ and $\{c,ad''-(a\circ d)' + a'd'\}$,
the second one (for $c>d$) coming from either of the rules above 
and $\{c,d''\} + 2\{c',d'\} + \{c'',d\} - [c,d]''$ in
\eqref{eq:IdealGenerators2}.

For example,
\[
ab\{c,d''\} \to f_1 = a\{c,(b\circ d)'\} - a\{c,b'\}d' -ab'\{c,d'\} - a[c,b]d''.
\]
The polynomial $f_1$ rewrites in $\mathcal G_4(B)$ as follows:
\[
f_1 \to ([c,a], b,d) - [c, (a,b,d)] + (a,[c,b], d).
\]
Here $(x,y,z)$ stands for $(x\circ y)\circ z - x\circ (y\circ z)$.
The expression obtained is symmetric with respect to $(a,b)$
by \eqref{eq:LSymm}
which means the convergence of the first type critical pair.

The second type critical pairs based on (A5) appear
when we rewrite 
\[
ab[c,d'']\to f_2 = ab\{d,c''\} - 2ab\{c',d'\} + ab[c,d]''.
\]
We already know how to rewrite all terms in the right-hand side to get an element in $\Bbbk B$. 
Comparing the result with what is obtained from $f_1$
we obtain a relation which is a corollary of 
\eqref{eq:GD1}, \eqref{eq:Spec1}, and \eqref{eq:Spec2}.
(Note that \eqref{eq:GD1} and \eqref{eq:Spec1}
allow us to rewrite both summands in $[c,(a,b,d)]$
via shorter terms, like those in \eqref{eq:Spec2}.)

In \cite{KSO2019} it was proved that 
\eqref{eq:Spec1} and \eqref{eq:Spec2} are independent identities 
on a GD-algebra.  Now we may state the following

\begin{proposition}
All special identities of $\GD$-algebras of degree  $\le 4$ are
consequences of (\ref{eq:Spec1}) and (\ref{eq:Spec2}).
\end{proposition}

In a similar way, we may find the complete list of special identities 
of degree 5.
The list of ambiguities in $\mathcal G_5(B)$ is relatively long since we have to consider all Poisson differential monomials
of degree 5 and weight~$-1$:
\[
\begin{gathered}{}
[a,[b,[c,d']]]e,\ [a,[b,c'']]de,\ 
[a,[b',c']]de,\ [a',[b,c']]de,\ 
[a,[b,c']]de',\ [a,b'][c,d']e, \\
[a,b^{(3)}]cde,\ [a',b'']cde,\ [a,b'']cde',\ [a',b']cde',\ [a,b']cde'',\ [a,b']cd'e'.
\end{gathered}
\]
As a result of the same computations as for $n=4$, we obtain 
three special identities of degree~5:
\begin{multline}\label{eq:Spec4}
[d\circ a,[b,e\circ c]]=[e\circ a,[b,d\circ c]]+[d,[b,e\circ c]\circ
a]-[d,[b,e]\circ a]\circ c+([d,[b,e]]\circ c)\circ a  \\
-[d,[b,e\circ c]]\circ a-[e\circ a,[b,d]\circ c]-[e,[b,d\circ c]]\circ 
a+[e,[b,d]\circ c]\circ a+[d\circ a,[b,e]\circ c] \\
+[d,[b,e\circ c]]\circ a-[d,[b,e]\circ c]\circ a-[e,[b,d\circ c]\circ 
a]+[e,[b,d]\circ a]\circ c-([e,[b,d]]\circ c)\circ a \\
+[e,[b,d\circ c]]\circ a,
\end{multline}
\begin{multline}\label{eq:Spec3}
[c,[a,e\circ b]\circ d]=[a,[c,e\circ d]\circ b]-[a,[c,e\circ d]]\circ 
b-[a,[c,e]\circ b]\circ d+([a,[c,e]]\circ d)\circ b \\
+[c,[a,e]\circ d]\circ b+[c,[a,e\circ b]]\circ d-([c,[a,e]]\circ d)\circ b,
\end{multline}
and 
\begin{multline}\label{eq:Spec5}
[a,d\circ b]\circ(c\circ e)=[a,c\circ b]\circ(d\circ e)+([a,d\circ b]\circ 
c)\circ e+([a,d]\circ(c\circ e))\circ b-(([a,d]\circ c)\circ e)\circ b \\
-([a,c\circ b]\circ d)\circ e-([a,c]\circ(d\circ e))\circ b+(([a,c]\circ d)\circ 
e)\circ b.
\end{multline}
Other critical pairs are convergent modulo  \eqref{eq:Spec1}, \eqref{eq:Spec2}, \eqref{eq:Spec4}, \eqref{eq:Spec3}, and \eqref{eq:Spec5}.
We do not state the details here since in the next section 
we show that in fact all special identities of degree $\le 5$
are corollaries of \eqref{eq:Spec1} and \eqref{eq:Spec2}.

\section{On the Gr\"obner basis of the Gelfand--Dorfman operad with special identities}\label{programm}

Let us recall the basic definitions related with operads following \cite[Chapter~5]{LodayVallette}.
A (symmetric) operad $\mathcal P$ in the category $\Vecc_\Bbbk $ of linear 
spaces over a field $\Bbbk $ is
a collection of spaces $\mathcal P(n)$, $n\geq 1$, 
equipped with linear composition maps
\[
\gamma^m _{n_1,\dots,n_m} :\mathcal P(m)\otimes \mathcal P(n_1)\otimes\ldots 
\otimes \mathcal P(n_m)\rightarrow \mathcal P(n_1+\ldots +n_m)
\]
for all integers $m,n_1,\ldots,n_m\geq 1$, 
each $\mathcal P(n)$ is a module over the symmetric group $S_n$.
These data have to satisfy the following conditions:
the composition is associative and equivariant relative to the 
action of $S_n$; the space $\mathcal P(1)$ contains an element $1$
which acts as an identity relative to the composition.

Every operad may be considered as an image of an 
appropriate free operad, i.e., a quotient 
modulo an operad ideal.

Namely, for every graded space 
$V = \bigoplus_{n\ge 1} V(n)$ there exists a uniquely 
defined free operad $\mathcal F(V)$ generated by $V$.
An operad ideal of $\mathcal F(V)$
may be presented as a minimal one that contains a given series of elements from
$\bigcup\limits_{n\geq 1} \mathcal F(V)(n)$.
Therefore, an operad may be defined by generators and relations.

For example, the operad $\Lie $ 
governing the variety of Lie algebras
is generated by $V_1=V_1(2)$, where $\dim V_1(2)=1$, 
this is a skew-symmetric $S_2$-module.
The free operad $\mathcal F(V_1)$ is exactly the operad 
of anti-commutative algebras. The set of defining relations 
of $\Lie $ consists of the Jacobi identity:
\[
\gamma^2_{2,1} (\mu,\mu,1) + \gamma^2_{2,1} (\mu,\mu,1)^{(123)} + \gamma^2_{2,1} 
(\mu,\mu,1)^{(132)}.
\]

The operad $\Nov $ of the variety of Novikov algebras 
is generated by $V_2 = V_2(2)$, $\dim V_2=2$, this is the 
regular $S_2$-module. Namely, a basis of 
$V_2(2)$ consists of $\nu$, $\nu^{(12)}$.
The free operad $\mathcal F(V_2)$ is exactly the magmatic one.
The defining relations of $\Nov $ include left symmetry and right
commutativity:
\[
\begin{gathered}
\gamma^2_{2,1} (\nu,\nu,1) - \gamma^2_{2,1} (\nu,\nu,1)^{(12)} - 
\gamma^{2}_{1,2} (\nu,1,\nu) 
+  \gamma^{2}_{1,2} (\nu,1,\nu)^{(12)}, \\
\gamma ^2_{1,2}(\nu,1,\nu) -  \gamma^2_{1,2} (\nu,1,\nu)^{(23)}.
\end{gathered}
\]

There is an intermediate notion between nonsymmetric and symmetric operads, known as a {\em shuffle operad}.
By definition, a shuffle operad $\mathcal V$ is a collection of spaces $\mathcal V(n)$, $n\ge 1$, equipped with 
a collection of compositions 
\[
 \gamma_\pi : \mathcal V(m)\otimes \mathcal V(n_1)\otimes\ldots 
\otimes \mathcal V(n_m)\rightarrow \mathcal V(n_1+\ldots +n_m)
\]
where 
$\pi $ is a {\em shuffle partition} of the set $\{1,\dots , n\}$, 
$n=m_1+\dots + m_n$, into $m$ disjoint subsets $I_j$, $j=1,\dots, m$, 
such that $\min I_1<\min I_2<\dots <\min I_m$.
Thus, a shuffle operad has no symmetric module structure, but its composition
structure still carries some information about the order of arguments.

Shuffle operads provide a convenient framework for the computation of 
Gr\"obner bases and normal forms 
in an operad defined by generators and relations \cite{DotKhor2010}. 
There is a forgetful functor $\mathcal P\mapsto \mathcal P^f$
that turns a symmetric operad $\mathcal P$ into a 
a shuffle one \cite[Section 5.3]{bremner-dotsenko} such that 
the normal form of elements in $\mathcal P^f(n)$ allows us to recover 
the normal form in $\mathcal P(n)$. In particular, we may find 
the dimensions of  $\mathcal P(n)$ in this way.

Following \cite{bremner-dotsenko}, we  may convert the symmetric operad $\GD$ into 
a shuffle operad $\GD^f$, a homomorphic image of the shuffle tree operad $\mathcal{T}_\Sha(\mathcal{X})$
for an appropriate language $\mathcal X$. 

Let us replace the operations $\nu $ and $\mu $ with the set of three operations 
$\mathcal X=\{x,y,z\}$, where $x$ and $y$ represent $\nu $ and $\nu^{(12)}$, 
$z$ represents~$\mu $. We will use the natural notation $x(1\;2) = \nu (x_1,x_2)$, 
$y(1\;2)=\nu(x_2,x_1)$, etc. 
To convert the defining relations of the operad $\GD$
into elements of
$\mathcal{T}_\Sha(\mathcal{X})$
we may use the equivariance property of the composition (see \cite{bremner-dotsenko} for details).
As a result, \eqref{eq:LSymm} and \eqref{eq:RComm} turn into the following relations for 
$x$ and $y$:
\[
\begin{gathered}
x(x(1\;2)\;3) - x(1\; x(2\; 3)) - x(y(1\; 2)\; 3) + y(x(1\; 3)\; 2), \\
x(x(1\; 3)\; 2) - x(1\; y(2\; 3)) - x(y(1\; 3)\; 2) + y(x(1\; 2)\; 3),\\
y(1\; x(2\; 3)) - y(y(1\; 3)\; 2) - y(1\; y(2\; 3)) + y(y(1\; 2)\; 3),\\
x(x(1\; 2)\; 3) - x(x(1\; 3)\; 2),\\
x(y(1\; 2)\; 3) - y(1\; x(2\; 3)),\\
x(y(1\; 3)\; 2) - y(1\; y(2\; 3)).
\end{gathered}
\]
The Jacobi identity may be expressed in terms of $z$ 
as
\[
z(z(1\; 2)\; 3) - z(1\; z(2\; 3)) - z(z(1\; 3)\; 2).
\]
Similarly, converting \eqref{eq:GD1} (namely, all relations in the symmetric group orbit of the 
defining relation), we obtain 
\[
\begin{gathered}
z(1\; x(2\; 3)) + z(y(1\; 2)\; 3) - x(z(1\; 2)\; 3) - y(1\; z(2\; 3)) - 
y(z(1\; 3)\; 2), \\
-z(x(1\; 3)\; 2) + z(x(1\; 2)\; 3) + x(z(1\; 2)\; 3) - x(z(1\; 3)\; 2) - x(1\; 
z(2\; 3)), \\
-y(z(1\; 2)\; 3) + z(1\; y(2\; 3)) + z(y(1\; 3)\; 2) - x(z(1\; 3)\; 2) + y(1\; 
z(2\; 3)).
\end{gathered}
\]
The elements ({\em shuffle tree polynomials} \cite{bremner-dotsenko}) obtained generate the ideal of defining relations for the operad 
$\GD^f$, a quotient of $\mathcal{T}_{\Sha}(\mathcal{X})$.

Calculating the Gr\"obner base of $\GD^f$ by means of the package \cite{DotsHij}, 
we get the following result for $\dim \GD(n)=\dim \GD^f(n)$.
\begin{center}
\begin{tabular}{c|cccccc}
 $n$ & 1 & 2 & 3 & 4 & 5 & 6 \\
 \hline 
 $\dim(\GD(n)) $ & 1 & 3 & 17 & 140 & 1524 & 20699
\end{tabular}
\end{center}

The first five terms of the sequence coincide with the number of certain planar graphs (see OEIS A322137, A291842). However, the sixth one is different. 
Finding a linear basis of the free $\GD$-algebra, or at least the sequence $\GD(n)$
is an interesting open problem. Note that the operad $\GD $ is not Koszul since so is $\Nov$ \cite{dzhuma1}.

By Corollary \ref{cor:SGD-Var}, the class $\SGD$ is defined by identities.
There should exist identities separating $\SGD$ from $\GD$, i.e., independent special identities of $\GD$-algebras. 
Let us consider the class of $\GD$-algebras with additional identities (\ref{eq:Spec1}) and 
(\ref{eq:Spec2}). Denote this class by $\wSGD$ (weak special Gelfand--Dorfman algebra).

As above, we may convert the defining relations of $\wSGD$ into shuffle tree polynomials  
by adding the orbits of relations (\ref{eq:Spec1}) and (\ref{eq:Spec2}) to $\GD^f$.
Namely, (\ref{eq:Spec1}) and (\ref{eq:Spec2}) give rise to the following elements of
$\mathcal{T}_{\Sha}(\mathcal{X})$:
\[
\begin{gathered}
z(1\; x(x(2\; 3)\; 4)) - x(z(1\; x(2\; 3))\; 4) - x(z(1\; x(2\; 4))\; 3) + 
x(x(z(1\; 2)\; 3)\; 4), \\
z(1\; x(y(2\; 3)\; 4)) - x(z(1\; y(2\; 3))\; 4) - x(z(1\; x(3\; 4))\; 2) + 
x(x(z(1\; 3)\; 2)\; 4), \\
z(1\; y(2\; y(3\; 4))) - x(z(1\; y(3\; 4))\; 2) - x(z(1\; y(2\; 4))\; 3) + 
x(x(z(1\; 4)\; 2)\; 3),\\
-z(x(x(1\; 3)\; 4)\; 2) + x(z(x(1\; 3)\; 2)\; 4) + x(z(x(1\; 4)\; 2)\; 3) - x(x(z(1\; 2)\; 3)\; 4), \\
-z(x(y(1\; 3)\; 4)\; 2) + x(z(y(1\; 3)\; 2)\; 4) - y(1\; z(2\; x(3\; 4))) + x(y(1\; z(2\; 3))\; 4), \\
-z(x(y(1\; 4)\; 3)\; 2) + x(z(y(1\; 4)\; 2)\; 3) - y(1\; z(2\; y(3\; 4))) + x(y(1\; z(2\; 4))\; 3), \\
-z(x(x(1\; 2)\; 4)\; 3) + x(z(x(1\; 2)\; 3)\; 4) + x(z(x(1\; 4)\; 3)\; 2) - x(x(z(1\; 3)\; 2)\; 4), \\
-z(x(y(1\; 2)\; 4)\; 3) + x(z(y(1\; 2)\; 3)\; 4) + y(1\; z(x(2\; 4)\; 3)) - y(1\; x(z(2\; 3)\; 4)), \\
-z(x(y(1\; 4)\; 2)\; 3) + x(z(y(1\; 4)\; 3)\; 2) + y(1\; z(y(2\; 4)\; 3)) + y(1\; y(2\; z(3\; 4))), \\
-z(x(x(1\; 2)\; 3)\; 4) + x(z(x(1\; 2)\; 4)\; 3) + x(z(x(1\; 3)\; 4)\; 2) - x(x(z(1\; 4)\; 2)\; 3), \\
-z(x(y(1\; 2)\; 3)\; 4) + x(z(y(1\; 2)\; 4)\; 3) + y(1\; z(x(2\; 3)\; 4)) - x(y(1\; z(2\; 4))\; 3), \\
-z(x(y(1\; 3)\; 2)\; 4) + x(z(y(1\; 3)\; 4)\; 2) + y(1\; z(y(2\; 3)\; 4)) - x(y(1\; z(3\; 4))\; 2)
\end{gathered}
\]
and
\begin{multline*}
z(x(1\; 2)\; x(3\; 4)) - x(z(x(1\; 2)\; 3)\; 4) - x(z(1\; x(3\; 4))\; 2) + 2 x(x(z(1\; 3)\; 2)\; 4) \\
+ z(x(1\; 4)\; y(2\; 3)) - x(z(1\; y(2\; 3))\; 4) - x(z(x(1\; 4)\; 3)\; 2),
\end{multline*}
\begin{multline*}
z(x(1\; 3)\; x(2\; 4)) - x(z(1\; x(2\; 4))\; 3) - x(z(x(1\; 3)\; 2)\; 4) + 2 x(x(z(1\; 2)\; 3)\; 4) \\
+ z(x(1\; 4)\; x(2\; 3))- x(z(1\; x(2\; 3))\; 4) - x(z(x(1\; 4)\; 2)\; 3),
\end{multline*}
\begin{multline*}
z(y(1\; 2)\; y(3\; 4)) - y(1\; z(2\; y(3\; 4))) - x(z(y(1\; 2)\; 4)\; 3) + 2 y(1\; x(z(2\; 4)\; 3))  \\
- z(y(1\; 4)\; x(2\; 3))+ x(z(y(1\; 4)\; 2)\; 3) - y(1\; z(x(2\; 3)\; 4)),
\end{multline*}
\begin{multline*}
z(y(1\; 2)\; x(3\; 4)) - y(1\; z(2\; x(3\; 4))) - x(z(y(1\; 2)\; 3)\; 4) + 2 y(1\; x(z(2\; 3)\; 4)) \\
- z(y(1\; 3)\; x(2\; 4)) + x(z(y(1\; 3)\; 2)\; 4) - y(1\; z(x(2\; 4)\; 3)),
\end{multline*}
\begin{multline*}
z(y(1\; 3)\; y(2\; 4)) + y(1\; z(y(2\; 4)\; 3)) - x(z(y(1\; 3)\; 4)\; 2) + 2 y(1\; y(2\; z(3\; 4))) \\
- z(y(1\; 4)\; y(2\; 3)) - y(1\; z(y(2\; 3)\; 4)) + x(z(y(1\; 4)\; 3)\; 2),
\end{multline*}
\begin{multline*}
z(x(1\; 2)\; y(3\; 4)) - x(z(1\; y(3\; 4))\; 2) - x(z(x(1\; 2)\; 4)\; 3) + 2 x(x(z(1\; 4)\; 2)\; 3) \\
+ z(x(1\; 3)\; y(2\; 4)) - x(z(x(1\; 3)\; 4)\; 2) - x(z(1\; y(2\; 4))\; 3).
\end{multline*}

Calculating the Gr\"obner base of the operad $\wSGD^f$ by means of the package \cite{DotsHij}, 
we get the following result:
\begin{center}
\begin{tabular}{c|ccccc}
 $n$ & 1 & 2 & 3 & 4 & 5 \\
 \hline 
 $\dim(\wSGD (n)) $ & 1 & 3 & 17 & 130 & 1219
\end{tabular}
\end{center}
The dimensions of $\SGD(n)$ were computed in \cite[Section~4]{KSO2019}. 
For $n\le 5$,  we have $\dim \SGD(n)=\dim \wSGD(n)$.
Hence, all special identities of degree $\le 5$ are corollaries of (\ref{eq:Spec1}) and (\ref{eq:Spec2}). We obtain the following result:

\begin{corollary}
The relations 
(\ref{eq:Spec4}), (\ref{eq:Spec3}) and (\ref{eq:Spec5}) are consequences 
of (\ref{eq:Spec1}) and (\ref{eq:Spec2}). 
\end{corollary}

It remains an open problem whether (\ref{eq:Spec1}) and (\ref{eq:Spec2}) exhaust all
independent special identities of GD-algebras.

\end{document}